\documentclass[11pt,a4paper]{article}
\usepackage[utf8]{inputenc}
\usepackage{amsmath}
\usepackage{amsthm}
\usepackage{import}
\usepackage{geometry}
\usepackage{mathrsfs} 
\usepackage{amsfonts}
\usepackage{amsmath}
\usepackage{amssymb}
\usepackage{palatino}
\usepackage{hyperref}
\usepackage{enumerate,graphicx}
\usepackage[usenames]{color}
\usepackage{comment}

\newtheorem{mydef}{Definition}[section]

\newtheorem{prop}{Proposition}[section]
\newtheorem{theorem}{Theorem}[section]
\newtheorem{lemma}{Lemma}[section]

\newtheorem{remark}{Remark}[section]

\newtheorem{assumption}[theorem]{Assumption}

\newtheorem{cor}{Corollary}[section]
\newcommand{\jent}{\mathscr{H}}
\newcommand{\ment}{\mathscr{H}}
\newcommand{\ecostt}{\cT_{R^{\ve}_{0t}}}
\newcommand{\ecostone}{\cT_{R^{\ve}_{01}}}
\newcommand{\ecostoneone}{\cT_{R_{01}}}
\newcommand{\otgrad}{\nabla^{W_2}}
\newcommand{\manif}{\RD}
\newcommand{\ei}{{\bf ETI}($\lambda,\ve,s,t$) }
\newcommand{\ti}{{\bf TI}($\lambda$) }

\newcommand{\tii}{{\bf TI}}
\newcommand{\eiuno}{{\bf ETI}($\lambda,\ve,s,1$) }
\newcommand{\eit}{{\bf ETI}($\lambda,\ve,t$) }
\newcommand{\eitd}{{\bf ETI}($2\lambda,\ve,t$) }
\newcommand{\eituno}{{\bf ETI}($\lambda,\ve,1$) }
\newcommand{\tle}{\theta_{\lambda\ve}}
\newcommand{\SP}{SP}

\newcommand{\R}{\mathbb{R}}

\newcommand{\RD}{\mathbb{R}^d}
\newcommand{\N}{\mathbb{N}}

\newcommand{\cM}{\mathcal{M}}
\newcommand{\cP}{\mathcal{P}}

\newcommand{\cT}{\mathcal{T}}

\newcommand{\bes}{\begin{equation*}}
\newcommand{\ees}{\end{equation*}}
\newcommand{\beas}{\begin{eqnarray*}}
\newcommand{\eeas}{\end{eqnarray*}}
\newcommand{\bea}{\begin{eqnarray}}
\newcommand{\eea}{\end{eqnarray}}
\newcommand{\be}{\begin{equation}}
\newcommand{\ee}{\end{equation}}

\newcommand{\ve}{\varepsilon}

\newcommand{\bbl}{\begin{block}}
\newcommand{\ebl}{\end{block}}

\newcommand{\De}{\mathrm{d}}

\title{Around the entropic Talagrand inequality}

\author{Giovanni Conforti\thanks{D\'epartement de Math\'ematiques Appliqu\'ees, \'Ecole Polytechnique, Route de Saclay, 91128, Palaiseau Cedex, France. giovanni.conforti@polytechnique.edu }, Luigia Ripani\thanks{Univ Lyon, Université Claude Bernard Lyon 1, CNRS UMR 5208, Institut Camille Jordan, 43 blvd. du 11 novembre 1918, F-69622 Villeurbanne cedex, France. ripani@math.univ-lyon1.fr}}

\begin{document}

\maketitle

\begin{abstract}
   In this article we study generalization of the classical Talagrand transport~-entropy inequality in which the Wasserstein distance is replaced by the entropic transportation cost. This class of inequalities has been introduced in the recent work \cite{conforti2017second}, in connection with the study of Schr\"odinger bridges.  We provide several equivalent characterizations in terms of reverse hypercontractivity for the heat semigroup, contractivity of the Hamilton-Jacobi-Bellman semigroup and dimension-free concentration of measure. Properties such as tensorization and relations to other functional inequalities are also investigated. In particular, we show that the inequalities studied in this article are implied by a Logarithmic Sobolev inequality and imply Talagrand inequality.
\end{abstract}

\tableofcontents

\section{Introduction and statements of the main results}

 A first probabilistic approach to transportation problems goes back to the early works of Schr\"odinger \cite{Schr,Schr32}, who was interested in finding the most likely evolution of a cloud of independent brownian particles towards a given ``unexpected" configuration. A rigorous formulation of Schr\"odinger's question is achieved through a constrained entropy minimization, known as the \emph{Schr\"odinger problem} (SP). The optimal value in (SP) measures intuitively the asymptotic probability that the particles attain the desired configuration, and is called \emph{entropic transportation cost}.  Mikami discovered in \cite{Mik04} (see also \cite{leonard2012schrodinger}) a fundamental connection with deterministic optimal transport, by showing that the Monge-Kantorovich problem (MK) may be seen as a ``small noise limit" of the Schr\"odinger problem. The study of the relations between these two transportation problems is nowadays an active field of research for at least two reasons: on the one hand the fact that (SP) provides with a regular convex approximation of (MK) has led to computational advantages \cite{cuturi2013sinkhorn, benamou2015iterative}; on the other hand the goal is understanding what is the ``stochastic" counterpart of the large body of results concerning the interplay between optimal transport, functional inequalities and curvature-like conditions \cite{gentil2015analogy, gentil2018dynamical, conforti2017second}. The present article contributes to this second line of research by studying a family of functional inequalities introduced in \cite{conforti2017second} which naturally generalizes Talagrand's transportation inequality \cite{talagrand1996transportation} to the entropic cost: for this reason we call them \emph{entropic Talagrand inequalities}. \\
We recall that a probability measure $m$ on $\RD$ satisfies Talgrand's transportation inequality with constant $C$, if for any probability measure $\mu$ we have

\be\label{talagrand} W_2^2(\mu,m) \leq C \ment(\mu|m), \ee
where $W_2^2(\cdot,\cdot)$ is the squared Wasserstein distance of order two and $\ment(\cdot|m)$ is the relative entropy w.r.t. $m$. This inequality was first introduced in \cite{talagrand1996transportation} for the Gaussian measure in the Euclidean space by Talagrand, and then generalized in \cite{otto2000generalization} by Otto and Villani.  
Later on we will adopt the notation \ti  for the classical Talagrand inequality \eqref{talagrand} with constant $C=1/\lambda$. 

To introduce the entropic version of \eqref{talagrand}, we fix a probability measure $m(\De x)=\exp(-2U(x))\De x$ and a noise parameter $\ve>0$ and consider the Langevin dynamics for $U$

\be\label{langevin} \De X_t  = -\ve \nabla U (X_t) \De t + \sqrt{\ve} \De B_t, \quad X_0 \sim m. \ee

Next, we call $R^{\ve}_{0t}$ the joint law at times ${0,t}$ of the Langevin dynamics: $R^{\ve}_{0t}$ acts as \emph{reference measure} to define the entropic transportation cost $\cT_{R^{\ve}_{0t}}(\mu,\nu)$ via the associated Schr\"odinger problem. The latter consists in minimizing the relative entropy w.r.t. the reference measure $R^{\ve}_{0t}$ over the set of couplings of $\mu$ and $\nu$. Leaving precise statements for later, let us just say that a probability measure $m$ on $\RD$ satisfies an entropic Talagrand inequality if

\[ \forall\, \mu, \quad \ecostt(\mu,m) \leq C   \ment(\mu|m) \]

or, more generally,

\[\forall\, \mu,\nu, \quad \ecostt(\mu,\nu)\leq C \ment(\mu|m)+ C'\ment(\nu|m).   \]

These inequalities are stronger than the classical Talagrand inequality since the entropic transport cost dominates the Wasserstein, see Remark \ref{domination} below. Moreover, the classical Talagrand inequality is recovered in the limit when $\ve \rightarrow 0$.  The main results of this article include equivalent characterizations of the entropic Talagrand inequalities in terms of a weak form of reverse hypercontractivity for the  semigroup associated with \eqref{langevin}, contractivity properties for the Hamilton-Jacobi-Bellmann semigroup and a dimension-free concentration property, in the spirit of \cite{gozlan2009characterization}; all these characterizations allow to recover well known results about Talagrand's inequality in the small noise limit. Furthermore, we show that the entropic Talagrand inequalities tensorize, and investigate relations with classical inequalities. In particular we extend Otto-Villani's Theorem \cite{otto2000generalization}, by showing that the entropic transportation inequality is implied by a Logarithmic Sobolev inequality, and that it implies the classical Talagrand's inequality. As a byproduct, we obtain that the entropic Talagrand inequalities hold under the celebrated Bakry-\'Emery $\Gamma_2$ condition \cite{BAKEM}. This fact has already been proven for measures on a compact Riemannian manifold in \cite{conforti2017second}.

 Transport-entropy inequalities for general costs have been studied in \cite{gozlan2017kantorovich} (and also \cite{bowles2018theory}). An observation we make here is that  (a slight modification of) the entropic cost is indeed one of those general costs. This allows us to profit from the results contained in \cite{gozlan2017kantorovich}, thus simplifying some of our proofs. Conversely, we provide a novel concrete example of functional inequality which can be treated with the methods of \cite{gozlan2017kantorovich}; moreover we  can provide explicit conditions for this inequality to hold, something which cannot be achieved for the general costs considered there. Finally, let us remark that, to streamline exposition, we limit ourselves to take $\RD$ as ambient space; however, it is very likely that the results we present here remain valid in a much wider setting.

\subsubsection*{Organization of the article}
We recall at Section 1 some basic facts about (SP) and its connections to optimal transport. In Section 2, we first introduce the class of entropic Talagrand inequalities at Definition \ref{ETIdef}, and prove two characterization results, Theorem \ref{reversehyp} and Theorem \ref{EItequivalentfroms}. Next, we investigate different forms of tensorization at Proposition \ref{tenso1}, \ref{tenso2} and \ref{tensogentil}. Then, we use these results to derive concentration of measure at Theorem \ref{concentration}. We establish at Corollary \ref{ottovillani} connections with the classical Talagrand inequality and the Logarithmic Sobolev inequality. Finally, at Corollary \ref{infconv} we show that an entropic Talagrand inequality implies an infimum convolution Logarithmic Sobolev inequality. The appendix collects some useful results which are behind most of the proof presented here.

\subsection{ Schr\"odinger problem and entropic transportation cost}
In order to define (\SP), we shall  first introduce a few notation. We fix a probability measure $m$ on $\RD$ whose density w.r.t. the Lebesgue measure is $\exp(-2 U(x))$, where $U$ is assumed to satisfy the minimal hypothesis which guarantee existence of a weak solution for the SDE \eqref{langevin}. { This is the case for instance when there exists some constant $c>0$ such that one of the following assumptions holds true:
\begin{enumerate}[(i)]
\item
$\lim _{ |x|\to \infty}U(x)= + \infty$ and $ \inf \{|\nabla U|^2- \Delta U/2\}>- \infty,$ or
\item
$-x\cdot \nabla U(x)\le c(1+|x|^2),$ for all $x\in\mathbb R^d.$
\end{enumerate}
See \cite[Thm.\,2.2.19]{Roy99} for the existence result under the assumptions (i) or (ii).}
For any $\ve>0$, we call $R^{\ve}$ the law of \eqref{langevin} on the space of continuous paths over $[0,+\infty]$ and for $t>0$ we denote $R^{\ve}_{0t}$ the law  of $R^{\ve}$ at times
$0$ and $t$ :

\[R^{\ve}_{0t}(\cdot) = R^{\ve}((X_0,X_t) \in \cdot). \]
For any measurable space $E$, we denote by $\cP(E)$ the space of probability measures over $E$ and for any $p,q\in \cP(E)$ $\Pi(p,q)$ is the set of couplings of $p$ and $q$; finally $\ment(q|p)$ is the relative entropy of $q$ w.r.t. $p$ defined as, 
$$
\ment(q|p)=\left\{\begin{array}{ll}
\int \log\frac{\De q}{\De p}\,\De q     &\; \textrm{if} \;q \ll p, \\
+\infty     & \textrm{otherwise.}
\end{array}\right.
$$
We are now in position to define (\SP).  Given two marginal laws $\mu,\nu \in \cP(\RD)$ and $\ve,t>0$, the (static) Schr\"odinger problem is the problem of finding the 
coupling of $\mu$ and $\nu$ which minimizes the relative entropy against $R^{\ve}_{0t}$,

\be\label{schrodprob}\tag{SP} \inf \{ \jent(\pi | R^{\ve}_{0t}): \pi \in \Pi(\mu,\nu) \}, \ee
 We call the optimal value in \eqref{schrodprob} the \emph{entropic transportation cost} between $\mu$ and $\nu$, and denote it $\ecostt(\mu,\nu)$.

As it is the case for the Wasserstein distance, the entropic transportation cost admits a dual formulation. It is known that if $\mu,\nu\in \mathcal P(\mathbb R^d)$ have finite relative entropy w.r.t. $m$, then
\be\label{kanto}
\ve\ecostt(\mu,\nu)=\ve \ment(\mu|m)+\sup_{\varphi\in C_b(\mathbb R^d)} \left\{ \int Q_t^\ve \varphi \,\De\mu - \int  \varphi \,\De\nu\right\}
\ee
where for all $t\geq0$, $x\in\RD$

\be\label{hjb}  Q_t^\ve\varphi(x) = \inf_{p \in \cP(\RD)}  \left\{ \int \varphi(y)  p(\De y) + \ve \ment(p|r^{\ve}_t(x,\cdot)) \right\}, \ee
where $x \mapsto r_{t}^{\ve}(x,\cdot) \in \cP(\RD)$ is the $m$-a.s. defined Markov kernel such that

\be\label{kerneldec} R^{\ve}_{0t}(\De x \De y) = m(\De x)r_t^{\ve}(x,\De y) \ee
The semigroup $(Q_t^\ve)_{t\geq0}$ is the Hamilton Jacobi Bellman (HJB) semigroup characterizing the vanishing viscosity solutions for the Hamilton Jacobi equation. Different proofs of \eqref{kanto} in more general contexts are by now available, see for instance \cite{MT06, gentil2015analogy, chen2014relation, gentil2018dynamical,gigli2018benamou}.  Introducing the linear semigroup $(P^{\ve}_t)_{t \geq 0}$ associated with \eqref{langevin} allows to give an alternative formulation of the HJB semigroup. We have

\be\label{hjb1}
Q_t^\ve\varphi(x) = -\ve\log P_t^\ve \exp(-\varphi/\ve)(x), \;\; x \in \mathbb R^d.
\ee
{ Note that \eqref{hjb1} follows from the dual representation of the entropy \eqref{eq:deco}.}

\subsection{The connection with optimal transport}

A fundamental fact is that one recovers (MK) from (SP) as a small noise (or, equivalently, short time) limit. This was first proven in \cite{Mik04} when the reference measure is a Brownian motion and in \cite{leonard2012schrodinger} in a more general case using $\Gamma$-convergence. In particular, those results imply that for all $\mu,\nu \in \mathcal P(\mathbb R^d)$ with second moment and relative entropy w.r.t $m$ finite,

\be\label{conv}
\lim_{\ve\to 0^+} \ve \ecostt(\mu,\nu)= \frac{W_2^2(\mu,\nu)}{2t}.
\ee
where $W_2(\mu,\nu)$ is Wasserstein distance of order two is defined for all $\mu,\nu \in \mathcal P(\mathbb R^d)$ with second moment as 
$$
W_2^2(\mu,\nu)=\inf_{\pi \in \Pi(\mu,\nu) }\int |x-y|^2\pi(\De x\De y) 
$$
Furthermore application of the Laplace principle \cite[Thm 4.3.1]{dembo2009large} yields 

\be\label{Laplaceprinc} \forall x \in \RD,\; \lim_{\ve \rightarrow 0}  Q_t^\ve\varphi(x) = \inf_{y \in  \RD} \{  \varphi(y) + \frac{1}{2t}|x-y|^2\}:=Q^{0}_t \varphi(x). \ee
Here $Q^{0}_t\varphi$ is nothing but the  Hopf-Lax semigroup that appears in the classical Kantorovich duality formula of optimal transport,

\be\label{OTkanto}
\frac{1}{2}W_2^2(\mu,\nu)=\sup_{\varphi\in C_b(\mathbb R^d)} \left\{ \int Q_t^0 \varphi \,\De\mu - \int  \varphi \,\De\nu\right\}.
\ee

In \cite{gozlan2017kantorovich} the authors study a general family of transportation costs. In particular, they look at costs which can be defined considering a measurable function $c:\manif \times \cP(\manif) \rightarrow [0,+\infty]$ and setting  

\be\label{gozlancosts} \cT_{c}(\nu|\mu) = \inf_{\pi \in \Pi(\mu,\nu)} \int_{\manif} c(x,p(x,\cdot)) \mu(\De x)          \ee
 where for $\pi \in \Pi(\mu,\nu)$ the map  $x \mapsto p(x,\cdot)$ is the ($\mu$-almost everywhere uniquely determined) probability kernel such that
 
 \[  \pi(\De x \De y) = \mu(\De x)p(x,\De y).     \]
We observe that if we subtract the marginal entropy of $\mu$ to the entropic transportation cost, then we fall in the set of costs \eqref{gozlancosts}. This simple fact allows us to take advantage of the results in \cite{gozlan2017kantorovich}. Inspired from their framework, we define 
\be\label{def forwardentropiccost} \ecostt(\nu |\mu) = \inf\left\{ \int_{\RD} \ment(p(x,\cdot) |r^{\ve}_t(x,\cdot)) \mu(\De x)  : \pi \in \Pi(\mu,\nu) \right\}, \ee
which is nothing but the cost \eqref{gozlancosts} with the choice 

\be\label{entropygozlan2} c(x,p) = \ment(p| r^{\ve}_t(x,\cdot)). \ee

\begin{lemma}\label{entropygozlan}
For all $\mu,\nu$ such that $\ment(\mu|m)<+\infty$ we have that\footnote{We adopt the standard convention that $+\infty-c = +\infty$, if $c<+\infty$}  

\be\label{entropygozlaneq} \ecostt(\mu,\nu)-\ment(\mu|m) = \ecostt(\nu |\mu), \ee

\end{lemma}

\begin{proof}
Assume that $\ecostt(\mu,\nu)<+\infty$. In this case, the conclusion follows from the decomposition of the entropy formula (see \cite[Thm.~2.4]{leonard2014some} or Lemma \ref{additiveppty} from the appendix), valid for all $\pi \in \Pi(\mu,\nu)$

\be\label{entropygozlan1} \ment(\pi | R^{\ve}_{0t}) = \ment(\mu|m) + \int_{\RD} \ment(p(x,\cdot) | r^{\ve}_t(x,\cdot))  \mu(\De x)   \ee
and by taking the infimum on both sides. On the other hand, if $\ecostt(\mu,\nu)=+\infty$, we find from \eqref{entropygozlan1} that 
 
 \[ \forall \pi \in \Pi(\mu,\nu), \quad  \ment(\mu|m) + \int_{\RD} \ment(p(x,\cdot) | r^{\ve}_t(x,\cdot))  \mu(\De x) =+\infty.   \]
Using the fact that $\ment(\mu|m)<+\infty$, we get that $\ecostt(\nu|\mu)=+\infty$ as well, which is the desired conclusion.
\end{proof}

\begin{remark}\label{rem}
Note that the entropic transportation cost is symmetric, and together with \eqref{entropygozlaneq} it implies
\be\label{remark}
\ecostt(\mu,\nu)= \ecostt(\nu|\mu)+\ment(\mu|m)=\ecostt(\mu|\nu)+\ment(\nu|m),
\ee
and taking $\mu=m$ (or equivalently $\nu=m$), 
\be\label{remarkm}
\ecostt(m,\nu)=\ecostt(\nu,m)=\ecostt(\nu|m)=\ecostt(m|\nu)+\ment(\nu|m).
\ee
\end{remark}

\begin{remark}\label{domination}
The entropic transportation cost is larger than the quadratic Wasserstein distance. Indeed, it follows from \cite[Corollary 5.13]{gentil2015analogy} and the Benamou-Brenier formula \cite{benamou2000computational} that for all $\ve >0,\mu,\nu \in \cP(\RD)$:

\[ 
\varepsilon \ecostone(\mu,\nu) \geq \frac{\ve}{2}\ment(\mu|m) +\frac{\ve}{2} \ment(\nu|m)+\frac{1}{2} W_2^2(\mu,\nu).  
\]

\end{remark}
\section{Entropic Talagrand inequality and properties}

The family of inequalities we consider in this article has been introduced in the recent article \cite{conforti2017second} where it was shown that, on a smooth compact manifold $M$ satisfying the Bakry \'Emery condition
\be\label{eq:Bakem}  \forall x \in M, \quad \mathfrak{Ric}_x +2 \mathbf{Hess}_x\, U \geq \lambda \, \mathbf{id} \ee
we have that for all
$\mu,
\nu \in \cP(M)$, $s \in (0,1)$ and $\ve>0$:

\be\label{EI1} \ecostone(\mu,\nu) \leq \frac{1}{1-\exp(-\lambda \ve s)}\ment(\mu|m) + \frac{1}{1-\exp(-\lambda \ve (1-s))}\ment(\nu|m).   \ee
In view of Lemma \ref{entropygozlan} and \eqref{remark}, the latter is equivalent to 

\be\label{EI2} \ecostone(\nu|\mu) \leq \frac{1}{\exp(\lambda \ve s)-1}\ment(\mu|m) + \frac{1}{1-\exp(-\lambda \ve (1-s))}\ment(\nu|m). \ee
  Also, observe that setting $\nu=m$ and optimizing over $s$ in \eqref{EI1} yields

\be\label{EI4} \ecostone(\mu,m) \leq \frac{1}{1-\exp(-\lambda \ve )}\ment(\mu|m).\ee
This motivates the following definition. 
{
\begin{assumption}\label{assu}
Let $m=\exp(-2U(x))\De x \in \cP(\mathbb{R}^d)$ with $U$ such that \eqref{langevin} admits a weak solution and let $R^\ve_{0t}$ the joint law at time $0$ and $t$ of the path measure associated to \eqref{langevin}.
\end{assumption}}
  
\begin{mydef}[Entropic Talagrand inequalities]\label{ETIdef}
{Let $m \in \cP(\mathbb{R}^d)$ be such that Assumption \eqref{assu} is satisfied} and fix $\lambda>0$, $0 \leq s < t$.

\begin{enumerate}

\item[(i)] We say that $m$ satisfies the entropic Talagrand inequality {\bf ETI}($\lambda,\ve,s,t$) if for all $\mu,\nu \in \cP(\mathbb{R}^d)$, 
$$
\ecostt(\mu,\nu)\leq \frac{1}{1-\exp(-\lambda\ve s)}\ment(\mu|m)+\frac{1}{1-\exp(-\lambda\ve(t-s))}\ment(\nu|m).
$$


\item[(ii)] We say that  $m$ satisfies the entropic Talagrand inequality {\bf ETI}($\lambda,\ve,t$) if for all $\mu \in \cP(\mathbb{R}^d)$,
$$
\ecostt(\mu,m)\leq \frac{1}{1-\exp(-\lambda\ve t)}\ment(\mu|m).
$$

\end{enumerate}
\end{mydef}

Let us recall that once the measure $m$ is fixed, the law $R^{\ve}_{0t}$ is uniquely determined as the two-times marginal of the Langevin dynamics \eqref{langevin}.

{\begin{remark}
It can be deduced from the Benamou-Brenier formulation of the entropic transportation cost (see e.g. \cite{gigli2018benamou}) that the function $t\mapsto t\ecostt(\mu,\nu)$ is increasing, it can be easily verified that {\bf ETI}($\lambda,\ve,s,\tau$) for some $\tau>0$ implies 
\be\label{hierarchy}
\ecostt(\mu,\nu)\leq \frac{\tau}{t}\frac{1}{1-\exp(-\lambda\ve s)}\ment(\mu|m)+\frac{\tau}{t}\frac{1}{1-\exp(-\lambda\ve(t-s))}\ment(\nu|m).
\ee
for all $0<t<\tau$. Therefore {\bf ETI}($\lambda,\ve,s,\tau$)implies {\bf ETI}($\lambda',\ve,s,t$), where the value of $\lambda'$
can be deduced from \eqref{hierarchy}.
\end{remark}}

\subsection{Equivalent forms of the entropic Talagrand inequalities}
In this section we state and prove several equivalent characterizations of \ei and \eit in terms of reverse hypercontractivity for the heat semigroup (Thm. \ref{reversehyp}) contractivity of the HJB semigroup (Thm. \ref{EItequivalentfroms}) and dimension-free concentration of measure (Thm.\ref{concentration}).

\subsubsection*{ A weak form of reverse hypercontractivity}

To recall the notions of hypercontractivity (\cite{Nel67}) and reverse hypercontractivity we first recall the definition of the heat semigroup $(P^{\ve}_t)_{t\geq0}$ associated with \eqref{langevin},

\[\forall f>0, \quad P^{\ve}_t f(x) := \int_{\RD} f(y) r^{\ve}_{t}(x, \De y),   \]
where $r^{\ve}_{t}(x, \De y)$ is the transition kernel for $R^{\ve}_{0t}$. Note that we have the scaling relation

\be\label{timechange} \forall \ve,t>0, x\in \RD,\, f>0, \quad P^{\ve}_t f(x) =  P^{1}_{\ve t} f(x).\ee
For $f\geq 0$, $p \in \R, p\neq 0$, we set

\be\label{pnorm}    \|f \|_p := \left(\int_{\manif} f^p \De m \right)^{1/p}. \ee

{ Note that we do not ask $p>1$. For an $f>0$ such that $\log f$ is integrable, the norm $\| f\|_{0}$ is defined by

\bes
\|f \|_0 = \exp \left( \int_{\RD} \log f \De m \right).
\ees}

\begin{mydef}[Hypercontractivity and reverse hypercontractivity]
Let $\lambda, \ve>0$. The semigroup $(P^{\ve}_t)_{t \geq 0}$ is $\lambda$-hypercontractive if for all $t>0$, $p>1$ and  $f>0$ it holds that

\[ \|P^{\ve}_tf \|_{q} \leq \|f \|_{p}, \quad \text{where}\quad \frac{q-1}{p-1} = e^{2\lambda \ve t}. \]
On the other hand, $\lambda$-reverse hypercontractivity is defined asking that for all $t>0$, $p<1$ and $f>0$,

\[\|P^{\ve}_t f \|_{q} \geq \|f \|_{p}, \quad \text{where}\quad \frac{q-1}{p-1} = e^{2\lambda\ve t}. \]

\end{mydef}
Next Theorem shows that the dual form of \ei encodes a weaker form of $\frac{\lambda}{2}$-reverse hypercontractivity; we recall that Gross established in \cite{gross1975logarithmic} equivalence between the logarithmic Sobolev inequality and hypercontractivity. The equivalence between (full) reverse hypercontractivity and Log Sobolev is also knwon \cite[Thm 3.3]{bakry1994hypercontractivite}. In the proof, and in the rest of the article, we take advantage of the notation

\be\label{thetadef} \theta_{\lambda \ve}(s):= \frac{1}{1-\exp(-\lambda \ve s)}. \ee

\begin{theorem}[\ei and reverse hypercontractivity]\label{reversehyp}
{ Let $m \in \cP(\mathbb{R}^d)$ be such that Assumption \eqref{assu} is satisfied.} For $t\geq 0$ the following are equivalent

\begin{enumerate}
    \item[i)] $m$ satisfies \ei for all $s\in [0,t]$.
    
    \item[ii)] For all $f > 0$ and $p,q\in [0,1)\times (-\infty,0]$ such that $\frac{q-1}{p-1} = \exp(\lambda \ve t)$ we have
   
\begin{equation}\label{reverse-hyp}
    \| P^{\ve}_t f \|_q \geq \|f \|_p.
\end{equation}  
    
\end{enumerate}
\end{theorem}

{ \begin{remark}
Because of the restrictions on $(p,q)$, we were not able to conclude that the weak form of reverse hypercontractivity of Theorem \ref{reversehyp} is equivalent to the log-Sobolev inequality. This would be true if the constraint $(p,q)\in(0,1)\times(-\infty,0)$ could be dropped, see the classical reference \cite[Thm 3.3]{bakry1994hypercontractivite}. However, we will see at Corollary~\ref{ottovillani} below that the weak reverse hypercontractivity implies a Poincar\'e inequality. 
\end{remark}}

\begin{proof}



In the proof we set for simplicity $t=1$. Inspired by \cite[Prop.~4.5]{gozlan2017kantorovich}, which generalizes some of the results in \cite{bobkov1999exponential}, we look for the dual formulation of \eiuno. First we rewrite it multiplying by $\ve$ as, 

\be\label{veei}
\forall \mu,\nu \in \cP(\RD), \quad \ve\ecostone(\mu,\nu)\leq \ve \tle(s)\ment(\mu|m)+\ve\tle(1-s)\ment(\nu|m).
\ee
The dual formulation \eqref{kanto} tells that (i)
is equivalent to say that for all $s\in(0,1)$, $\varphi \in C_b(\RD)$, $\mu,\nu \in \cP(\RD)$ we have
\[ \ve \ment(\mu|m) + \int Q_t^\ve \varphi \,\De\mu - \int  \varphi \,\De\nu \leq  \ve \tle(s)\ment(\mu|m)+\ve\tle(1-s)\ment(\nu|m).\]
Rearranging the terms, we can rewrite the latter as
\beas
\ve(\tle(s)-1)\left(\int \frac{Q_1^\ve\varphi}{\ve(\tle(s)-1)} \De\mu-\ment(\mu|m) \right)\\
+\ve\tle(1-s)\left(\int -\frac{\varphi}{\ve\tle(1-s)}\De\nu-\ment(\nu|m)   \right)\leq 0
\eeas
 We now take the suprema over $\mu$ and $\nu$ and use the variational formula \eqref{HJBduality2}, to obtain that (i) is equivalent to the fact that for all $s\in(0,1)$ and $\varphi \in C_b(\RD)$
\beas
\ve(\tle(s)-1)\log \int \exp\left(\frac{Q_1^\ve\varphi}{\ve(\tle(s)-1)}\right) \De m \\
+ \ve\tle(1-s)\log \int \exp \left(-\frac{\varphi}{\ve\tle(1-s)}\right) \De m \leq 0.
\eeas
Taking exponentials we get

\be\label{dualeit}
\left(\int \exp\left(\frac{Q_1^\ve\varphi}{\ve(\tle(s)-1)}\right) \De m \right)^{\ve(\tle(s)-1)}\left(\int \exp(-\frac{\varphi}\ve\tle(1-s)) \De m \right)^{\ve\tle(1-s)}\leq 1. 
\ee
Using \eqref{hjb1} and setting $\exp(-\varphi/\ve)=f$ we obtain
$$ 
\left(\int (P_1^\ve f)^{-1/(\tle(s)-1)}\De m\right)^{\ve(\tle(s)-1)}\left(\int f^{1/\tle(1-s)}\De m \right)^{\ve\tle(1-s)}\leq 1.
$$
Raising to the power of $1/\ve$, using \eqref{pnorm} and setting $q(\lambda \ve,s)=-1/(\tle(s)-1)$, $p(\lambda \ve,s)=1/\tle(1-s)$ we obtain a new equivalent formulation of (i) after a simple approximation argument:

\[  \forall s \in(0,1),f>0, \quad \| P^{\ve}_1 f \|_{q(\lambda \ve,s)} \geq \| f \|_{p(\lambda \ve,s)}. \]
To conclude the proof, we first observe that

\[ \Big\{ (p(\lambda \ve,s),q(\lambda \ve,s)): s\in (0,1) \Big\} =\Big\{(p,q)\in (0,1) \times (-\infty,0): \frac{q-1}{p-1} = \exp(\lambda \ve) \Big\}. \]
{ The case $p=0$ is obtained with a standard approximation argument.}
\end{proof}

The dual formulation of \ti is equivalent to some contraction properties for the Hopf-Lax semigroup, see \cite[Prop 9.2.3]{bakry2013analysis}. Here we show that \eit admits a dual formulation in terms of contraction properties for the HJB semigroup. 

\begin{theorem}[\eit and the HJB semigroup]\label{EItequivalentfroms}
{ Let $m \in \cP(\mathbb{R}^d)$ be such that Assumption \eqref{assu} is satisfied.} The following are equivalent
\begin{enumerate}
    \item[(i)] \eit  holds;
    \item[(ii)] For all $\varphi \in C_b(\RD)$,
  
    $$
    \int\exp\left(-\frac{1}{\ve\tle(t)}\varphi\right)\De m\leq \exp\left( -\frac{1}{\ve\tle(t)} \int Q_t^\ve \varphi \De m\right);
    $$
    \item[(iii)] For all $\psi \in C_b(\RD)$,
    \be\label{HJBcontraction1}  \int \exp\left( Q^{\ve/C}_{ C t}\, \psi \right) \De m \leq  \exp\left(\int \psi \De m \right)   \ee
where 
\[C=\ve(\tle(t)-1). \]
\end{enumerate}
\end{theorem}

Remark that letting $\ve \rightarrow 0$ in \eqref{HJBcontraction1} gives back, at least formally, the above mentioned characterization of \ti.

\begin{proof}
We follow the same arguments as in the proof of Theorem~\ref{reversehyp}. Again, w.l.o.g. we fix $t=1$.  To prove $(ii)$, we multiply \eituno by $\ve$ and recall that according to the Kantorovich dual formulation \eqref{kanto} and the symmetric property for the entropic cost we have, 
$$
\ve\ecostone(\mu,m)= \sup_{\varphi \in C_b(\RD)}\left\{ \int Q_1^\ve \varphi \De m-\int \varphi \De\mu \right\}. 
$$
Plugging this into \eituno yields the equivalent formulation
$$
\forall \varphi \in C_b(\RD), \quad \int Q_1^\ve \varphi \De m -\int \varphi \De \mu -\ve \theta_{\lambda\ve}(1)\ment (\mu|m)\leq 0.
$$
This can be re-written as, 
$$
\forall \varphi \in C_b(\RD), \quad \frac{1}{\ve\theta_{\lambda\ve}(1)}\int Q_1^\ve\varphi \De m+\left(\int -\frac{\varphi}{\ve\theta_{\lambda\ve}(1)}\De \mu-\ment(\mu|m)\right)\leq 0.
$$
Taking the supremum over $\mu$ and exponentiating, we obtain the desired result thanks to \eqref{HJBduality2}. The proof of $(iii)$ is analogue. We start from the Kantorovich formulation of the entropic cost \eqref{kanto} to obtain that \eituno is equivalent to the property that for all $\varphi \in C_b(\RD)$ and $\mu \in \cP(\RD)$,
\[
 \ve\ment(\mu|m) +\sup_{\varphi \in C_b(\RD)} \left\{ \int Q^{\ve}_1\varphi\De \mu -\int\varphi \De m \right\} \leq \ve \theta_{\lambda \ve}(1) \ment(\mu|m).
\]
Rearranging terms, taking sumpremum over $\mu$, using \eqref{HJBduality2} we arrive at the following equivalent form of \eituno

$$
\forall \varphi \in C_b(\RD), \quad \ve(\theta_{\lambda\ve}(1)-1)\log\int\exp\left(\frac{Q_1^\ve\varphi}{\ve(\theta_{\lambda\ve}(1)-1)}\right) \De m -\int \varphi \De m\leq 0. 
$$
The conclusion follows by exponentiating, setting $\psi= \varphi/C$ and an application of the scaling relation (see \eqref{timechange})

$$
\frac{1}{C}Q_t^{\ve} (C \psi) =Q_{tC}^{\ve/C}(\psi).
$$

\end{proof}

\subsection{Properties of entropic Talagrand inequalities}
In the next lines, we investigate tensorization  of \eit and \ei. In what follows we adopt the following convention: if $p(x,\cdot)$ is a probability kernel on $\R^{d_1} \times \ldots\times \R^{d_n}$ we write $p_i(x,\cdot) \in \cP(\R^{d_i})$ for the $i$-$th$ marginal distribution of $p(x,\cdot)$.

\begin{prop}[Tensorization: first form]\label{tenso1}
Let $n \in \N$, $1 \leq i \leq  n$ and $m_i(\De x) =\exp(-2U_i(x))\De x \in \cP(\R^{d_i})$  such that Assumption \eqref{assu} is satisfied  and satisfy \ei. Then $m= m_1 \otimes \ldots \otimes m_n$ satisfies  \ei.
\end{prop}

{ We recall that in the above proposition  the entropic cost for $m_1 \otimes \ldots \otimes m_n$ is the one corresponding to the law of $n$ independent diffusions of the form \eqref{langevin} associated with the potentials $U_i$, $i=1,\ldots,n$.}

\begin{proof}
We assume again w.l.o.g. that $t=1$. Recall that \eiuno for $m_i$ has the equivalent form

\[ \mathcal{T}_{R^{\ve,i}_{01}}(\nu|\mu) \leq (\tle(s)-1)\ment(\mu|m_i) + \tle(1-s)\ment(\nu|m_i) \]
where $R^{\ve,i}_{01}$ is the two times law of the Langevin dynamics for $m_i$. By induction, it is also enough to consider only the case $n=2$.
Consider now $\mu,\nu \in \cP(\R^{d_1+d_2})$, and assume that 
$\ecostone(\mu,\nu)<+\infty$. Then, there exist an optimal kernel $\tilde{p}_1(x_1, \De y_1)$ such that 

\[  \int_{\R^{d_1}} \ment( \tilde{p}_1(x_1, \cdot ) | r^{\ve,1}_1(x_1,\cdot))  \mu_1(\De x_1) = \mathcal{T}_{R^{\ve,1}_{01}}(\nu_1|\mu_1)      \]
  where we denoted $\mu_1,\nu_1 $ the image laws of $\mu$ and $\nu$ through the projection on the first $d_1$ coordinates.\\
 Moreover, for any fixed $x_1,y_1 \in \R^{d_1}\times \R^{d_1}$ there exist an optimal kernel $q^{x_1,y_1}(x_2,\De y_2)$ on $\R^{d_2}$ such that 
 
 \[ \int_{\R^{d_2}}  \ment( q^{x_1,y_1}(x_2, \cdot ) | r^{\ve,2}_1(x_2,\cdot)) \mu(x_1,\De x_2) =   \mathcal{T}_{R^{\ve,2}_{01}}(\nu(\cdot | y_1)|\mu(x_1,\cdot ))            \]
where $\mu(x_1,\cdot)$ ( resp. $\nu(y_1,\cdot)$ ) is the kernel defined via $\mu(\De x_1 \De x_2) = \mu_1(\De x_1) \mu(x_1, \De x_2)$ (resp. $\nu(\De y_1 \De y_2) = \nu_1(\De y_1) \nu(y_1, \De y_2)$).\\
We can construct a coupling $\pi$ of $\mu$ and $\nu$ setting, 

\be\label{tens2} \pi(\De x \De y) =  
\mu(\De x_1 \De x_2 ) p(x,\De y)  \quad \, p(x,\De y)= \tilde{p}_1(x_1,\De y_1) q^{x_1,y_1}(x_2, \De y_2).  \ee 
Note that for any $x$ we have $p_1(x,\cdot)=\tilde{p}_1(x_1,\cdot )$ and

\[ p_2(x,\cdot) = \int_{\R^{d_1}} \tilde{p}_1(x_1,\De y_1) q^{x_1,y_1}(x_2, \cdot). \]
Since the Langevin dynamics for $m_1 \times m_2$ is the product of the Langevin dynamics for $m_1$ and $m_2$ we have

\bes
\ecostone(\mu,\nu) \leq \int_{\R^{d_1+d_2}} \ment( \tilde{p}_1(x_1,\cdot) q^{x_1,y_1}(x_2, \cdot)| r^{\ve,1}_{1}(x_1,\cdot) \otimes  r^{\ve,2}_{1}(x_2,\cdot) ) \mu(\De x_1 \De x_2).
\ees
Thanks to the decomposition of the entropy formula \eqref{eq:deco} we have for all $\mu$ almost all $x_1,x_2$

\beas \ment( \tilde{p}_1(x_1,\cdot) q^{x_1,y_1}(x_2, \cdot)| r^{\ve,1}_{1}(x_1,\cdot) \otimes  r^{\ve,2}_{1}(x_2,\cdot) ) = \ment ( \tilde{p}_1(x_1,\cdot) |r^{\ve,1}_{1}(x_1,\cdot)) \\
+ \int_{\R^{d_1}}  \ment(q^{x_1,y_1}(x_2, \cdot)| r^{\ve,2}_1(x_2,\cdot)) \tilde{p}_1(x_1,\De y_1) .     \eeas
Plugging this into the above formula  and using the optimality of the couplings yields

\beas
\ecostone(\mu,\nu) \leq  \mathcal{T}_{R^{\ve,1}_{01}}(\nu_1|\mu_1) \\
+\int_{\R^{d_1+ d_1}}\left[ \int_{\R^{d_2}}  \ment(q^{x_1,y_1}(x_2, \cdot)| r^{\ve,2}_1(x_2,))\mu(x_1,\De x_2) \right] \tilde{p}_1(x_1,\De y_1) \mu_1(\De x_1) \\
=\mathcal{T}_{R^{\ve,1}_{01}}(\nu_1|\mu_1) + \int_{\R^{d_1+ d_1}}\mathcal{T}_{R^{\ve,2}_{01}}(\nu(y_1,\cdot )|\mu(x_1,\cdot)) \tilde{p}_1(x_1,\De y_1) \mu_1(\De x_1).
\eeas
Applying \eiuno  and the fact that $\mu_1(x_1)\tilde{p}_1(x_1,\De y_1) = \nu_1(\De y_1)$ we get

\beas
\ecostone(\mu,\nu) \leq  (\tle(s)-1)\left[ \ment(\mu_1|m_1)+ \int_{\R^{d_1+ d_1}}\ment(\mu(x_1,\cdot)|m_2) \tilde{p}_1(x_1,\De y_1) \mu_1(\De x_1)  \right] +\\
\tle(1-s)\left[ \ment(\nu_1|m_1)+ \int_{\R^{ d_1+ d_1}}\ment(\nu(y_1,\cdot )|m_2)  \tilde{p}_1(x_1,\De y_1) \mu_1(\De x_1) \right] \\
= (\tle(s)-1) \left[ \ment(\mu_1|m_1)+ \int_{\R^{ d_1}}\ment(\nu(\cdot | y_1)|m_2)  \mu_1(\De x_1) \right] \\
+ \tle(1-s)\left[ \ment(\nu_1|m_1)+ \int_{\R^{ d_1}}\ment(\nu(\cdot | y_1)|m_2)  \nu_1(\De y_1) \right]\\
= (\tle(s)-1) \ment(\mu|m_1 \otimes m_2) +\tle(1-s) \ment(\nu|m_1 \otimes m_2)
\eeas
where the last equality follows from the decomposition of the entropy formula \eqref{eq:deco}. 

\end{proof}
 
 A second form of tensorization holds, following \cite{gozlan2017kantorovich}. 

\begin{prop}[Tensorization: second form]\label{tenso2}
Let $n \in \N$, $1 \leq i \leq  n$ and $m_i \in \cP(\R^{d_i})$ { such that Assumption \eqref{assu} is satisfied}  and satisfy \ei. Then $m= m_1 \otimes \ldots \otimes m_n$ satisfies  the following inequality

\be\label{tens1}
\forall \mu, \nu \in \mathcal P(\mathbb R^{d_1+\dots +d_n})\quad \bar{\mathcal{T}} (\nu|\mu) \leq  (\tle(s)-1) \ment(\mu|m) + \tle(1-s) \ment(\nu|m),
\ee
where 

\[\bar{\mathcal{T}}(\nu|\mu) = \inf_{\pi \in \Pi(\mu,\nu)} \int \sum_{i=1}^n \ment( p_i(x_i,\cdot) |  ) \, \mu(\De x ). \]

\end{prop}

\begin{proof}
The proof follows the same lines as the former one. As before, we can restrict to $n=2$, and construct the coupling 
$\pi$ via \eqref{tens2}. Note that 

\[ p_2(x,\cdot) = \int_{\R^{d_1}} \tilde{p}_1(x_1,\De y_1) q^{x_1,y_1}(x_2, \cdot) \]
 We have
 
 \beas 
\bar{\mathcal{T}}(\nu|\mu) \leq \int_{\R^{d_1+d_2}} \ment(p_1(x,\cdot)|r^{\ve,1}_{1}(x_1,\cdot)) \mu(\De x) + \int_{\R^{d_1+d_2}} \ment(p_2(x,\cdot)|r^{\ve,2}_{1}(x_2,\cdot)) \mu(\De x)\\
=\int_{\R^{d_1+d_2}} \ment(\tilde{p}_1(x_1,\cdot)|r^{\ve,1}_{1}(x_1,\cdot)) \mu_1(\De x_1)  \\
+\int_{\R^{d_1+d_2}} \ment\left(\int_{\R^{d_1}} \tilde{p}_1(x_1,dy_1)q^{x_1,y_1}(x_2,\cdot)|r^{\ve,2}_{1}(x_2,\cdot)\right) \mu(x_1,\De x_2) \mu_1(\De x_1)\\
\leq \mathcal{T}_{R^{\ve,1}_{01}}(\nu_1|\mu_1)  \\
+ \int_{\R^{d_1+ d_1+d_2}} \ment\left( q^{x_1,y_1}(x_2,\cdot)|r^{\ve,2}_{1}(x_2,\cdot)\right) \mu(x_1,\De x_2) \tilde{p}_1(x_1,dy_1) \mu_1(\De x_1) \\
= \mathcal{T}_{R^{\ve,1}_{01}}(\nu_1|\mu_1) + \\
 \int_{\R^{d_1+ d_1}}\mathcal{T}_{R^{\ve,1}_{01}}(\nu(y_1,\cdot)|\mu(x_1,\cdot))  \tilde{p}_1(x_1,\De y_1) \mu_1(\De x_1),
 \eeas
 { where the last inequality is a consequence of the convexity of the relative entropy.}
From now on, the proof goes as in the former proposition.
\end{proof}
It can be easily seen that Propositions~\ref{tenso1} and~\ref{tenso2} are valid also for \eit. However, we propose here an alternative proof of the tensorization property for \eit, in the same spirit of \cite[Prop.~9.2.4]{bakry2013analysis}. 

\begin{prop}[Tensorization: third form]\label{tensogentil}
Let $n \in \N$, $1 \leq i \leq  n$ and $m_i \in \cP(\R^{d_i})$ { such that Assumption \eqref{assu} is satisfied}  and satisfy \eit. Then $m_1\otimes \ldots \otimes m_n$ satisfies \eit.
\end{prop}

\begin{proof}
For any $\ve>0$, let $P^{\ve}_t$, $Q^{\ve}_t$ be the heat and HJB semigroups for $m_1 \times m_2$.
Also, we note $P^{\ve,1}_t$(resp.  $P^{\ve,2}_t$) and $Q^{\ve,1}_t$ (resp.  $Q^{\ve,2}_t$) the same semigroups for $m_1$ (resp. $m_2$). To obtain the result, we show that the equivalent form (iii) in Theorem~\ref{EItequivalentfroms} of \eit holds. To this aim, we observe that, thanks to the fact that the Langevin dynamics for $m_1 \times m_2$ is the product of the Langevin dynamics for $m_1$ and $m_2$, we have for all $\ve,t>0$ and $x_1,x_2 \in \RD$:

\be\label{semigroupnest}  Q^{\ve/c}_{ct} \varphi (x_1,x_2) = Q^{\ve/c,1}_{ct} \left( Q^{\ve/c,2,\cdot}_{ct}\psi \,(x_2) \right)(x_1) \ee
where, for any $(y_1,x_2) \in \RD \times \RD$,
\[ Q^{\ve/c, 2,y_1}_{ct}\psi(x_2) = Q^{\ve/c,2}_{ct}(\psi(y_1,.)) ( x_2). \]
Using \eqref{semigroupnest} and  \eit for $m_1$ we obtain,

\beas
\int  \exp(Q^{\ve/c}_{ct}) \varphi (x_1,x_2) \, m_1 \otimes m_2(\De x_1 \De x_2) \\
\leq \int \exp\left(  \int  Q^{\ve/c,2,x_1}_{ct}\psi(x_2)  \,   \mu(\De x_1) \right) \mu(\De x_2)
\eeas
Using the definition of $Q^{\ve/c,2}_{ct}$ as an infimum \eqref{hjb}, we obtain
\[\int  Q^{\ve/c,2,x_1}_{ct}     \mu(\De x_1)  \leq Q^{\ve/c,2}_{ct}\left( \int\psi(x_1,\cdot)\mu(\De x_1) \right)\, (x_2) . \]
Using this and \eit for $m_2$ we get
\beas
\int \exp\left(  \int  Q^{\ve/c,2,x_1}_{ct}\psi(x_2)     m_1(\De x_1) \right) m_2(\De x_2) \\
\leq \int \exp\left(   Q^{\ve/c,2}_{ct}\left( \int\psi(x_1,\cdot)m_1(\De x_1) \right)(x_2)      \right) m_2(\De x_2) \\
\leq\exp\left(\int\psi(x_1,x_2)m_1 \otimes m_2(\De x_1 \De x_2)    \right)  
\eeas
which is the desired conclusion.

\end{proof}

The tensorization property allows us to give a further characterization of \ei via a dimension free concentration property. Let us first introduce some notation. For $m\in \mathcal P(\mathbb R^d)$ we denote $m^n=\underbrace{m \otimes \ldots \otimes m}_{n\:\textrm{times}} \in \mathcal P(\mathbb R^{d\times n})=\mathcal P(\underbrace{\mathbb R^d\times \ldots \mathbb \times \R^d}_{n\; \textrm{times}})$; for any $t>0$ $R_{0t}^{\ve,n}$ is the joint law of the reference measure with reversing measure $m^n$ and generator $\mathscr{L}^{\ve,n}=\mathscr{L}^{\ve}\oplus \ldots \oplus \mathscr{L}^\ve$, $r_t^{\ve,n}(x,\cdot)$ its Markov kernel and $(P^{\ve,n}_t)_{t\geq0}$ the associated product Markov semigroup. Finally, in accordance with what we did above we define the corresponding HJB semigroup:

\[ Q^{\ve,n}_t\varphi(x)=- \ve \log P_t^{\ve,n}\exp(-\varphi/\ve) (x),\quad \text{for} x\in\mathbb{R}^{d\times n}.\]
For any Borel set $A\subset \mathbb{R}^{d\times n}$ following \cite{gozlan2017kantorovich} we consider 

\be\label{can}
c_A^n(x):=\inf \{\ment(p|r_1^{\ve,n}(x,\cdot)),\, p\in \mathcal P(\mathbb R^{d\times n}), \, p(A)=1 \}, \quad x\in \mathbb R^{d\times n}.
\ee
A standard calculation shows that

\[ c_A^n(x) = -\log r_1^{\ve,n}(x,A) . \]
Moreover, we define for all $u\geq0$, 
\be\label{au}
A_u :=\{x\in \mathbb R^{d\times n} : c_A^n(x)\leq u\} = \{x\in \mathbb R^{d\times n} : r_1^{\ve,n}(x,A) \geq e^{-u} \}.
\ee

\begin{remark}
Note that $A_u$ is not  in general an enlargement of $A$, i.e. $A \nsubseteq A_u$.
\end{remark}

In the next theorem we provide an equivalent characterization of \eiuno in terms of dimension-free concentration. Note that the tensorization result we use here is Proposition \ref{tenso1} and not Proposition \ref{tenso2}, as it is more natural in this context. Thus, our Theorem \ref{concentration} is close in spirit, but different from Theorem 5.1 in \cite{gozlan2017kantorovich}.

\begin{theorem}[Dimension free concentration]\label{concentration}
Let $R^\ve$ be the stationary Markov process for the generator $\mathscr L^\ve$ and {$m \in \cP(\mathbb{R}^d)$ such that Assumption \eqref{assu} is satisfied}. The following are equivalent for $\lambda>0$ and $s\in [0,1]$, 
\begin{enumerate}
    \item[(i)] $m$ satisfies \eiuno;
    \item[(ii)] For any integer $n\geq1$, for all Borel set $A\subset \mathbb{R}^{d\times n}$ and any $u\geq0$ it holds, 
    $$
    m^n(\mathbb R^{d\times n}\setminus A_u^n)^{\theta_{\lambda\ve}(s)-1}\, m^n(A)^{\theta_{\lambda\ve}(1-s)} \leq e^{-u},
    $$
    with $\theta_{\lambda\ve}(s)$ defined at \eqref{thetadef}.
    
    \item[(iii)] For all integers $n\geq1$, for all non-negative $\varphi \in C_b(\mathbb R^{d\times n})$, it holds, 
    $$
    m^n(Q^{\ve,n}_1\varphi>u)^{\ve \,( \theta_{\lambda\ve}(s)-1)}m^n(\varphi\leq v)^{\ve \, \theta_{\lambda\ve}(1-s)}\leq e^{v-u}, 
    $$
    for all  $v\in\mathbb R$ and $u$ s.t. $u-v>0$.
\end{enumerate}
\end{theorem}

\begin{proof}
The proof follows the one of \cite[Thm.~5.1]{gozlan2017kantorovich}. For completeness we recall here some key points. The implication $(i)\Rightarrow(ii)$ is a generalization to the entropic transportation inequality of Marton's argument. Since $m$ satisfies \eiuno then thanks to Prop.~\ref{tenso1} the same holds for $m^n$.  As observed at Remark~\ref{rem}, \eiuno can be equivalently written as
$$
\mathcal T_{R_{01}^{\ve,n}}(\nu|\mu) \leq (\theta_{\lambda\ve}(s)-1)\ment(\mu|m^n) + \theta_{\lambda\ve}(1-s)\ment(\nu|m^n)
$$
for all $\mu,\nu\in\mathcal P(\mathbb R^{d\times n})$. For $A\subset \mathbb R^{d\times n}$ we choose the couple of probability measures $\mu(\De x)=\mathbf{1}_B/m^n(B)m^n(\De x)$ and $\nu(\De x)=\mathbf{1}_A/m^n(A)m^n(\De x)$ where $B=\mathbb R^{d\times n}\setminus A_u$ and $A_u$ is defined at \eqref{au}. Hence direct computations show that $\ment(\mu|m^n)=-\log m^n(B)$ and $\ment(\nu|m^n)=-\log m^n(A)$.
Also, observe that the infimum value in \eqref{can} can be easily computed, providing $c_A^n(x)=-\log r_1^{\ve,n}(x,A).$ Moreover the set $A_u$ can be rewritten as,
$$
A_u = \{x\in \mathbb{R}^{d\times n} : r_1^{\ve,n}(x,A)\geq e^{-u} \}.
$$
To conclude, take any $\pi\in \Pi(\mu,\nu)$ with disintegration kernel $(p_x)_{x\in \mathbb R^{d\times n}}$ then 
$$
\int \ment(p_x|r_1^{\ve,n}(x,\cdot))\mu(\De x) \geq \int c_A^n(x) \mu(\De x)>u.
$$
 The conclusion follows by taking the infimum on the set of couplings of $\mu$ and $\nu$.\\
 For the implication $(ii)\Rightarrow(iii)$, let $\varphi\in C_b(\mathbb R^{d\times n})$ and consider $A=\{\varphi \leq v\}$ for some real $v$. We show that $\{Q_1^{\ve,n}\varphi>u\}\subset\{c_A^n>u-v\}$. Take $x\in \{Q_1^{\ve,n}\varphi>u\}$, then for all $p\in\mathcal P(\mathbb R^{d\times n})$ with $p(A)=1$, and thanks to \eqref{hjb} it holds, 
 $$
 u< \int \varphi \De p +\ve\ment(p|r_1^{\ve,n}(x,\cdot)) \leq v+ \ve\ment(p|r_1^{\ve,n}(x,\cdot)).
 $$
 The conclusion follows by optimizing among all the probability $p\in \mathcal P(\mathbb R^{d\times n})$ such that $p(A)=1$.
 To show the last implication $(iii) \Rightarrow (i)$ we fix for simplicity $n=2$. Let $\delta\in(0,1)$, $f$ a non-negative function on $\mathbb R^d$. Define $\varphi(x)=f(x_1)+f(x_2), x\in \mathbb R^{2\times d}$. Then according to \cite{bakry2013analysis} it can be verified that 
 $Q_1^{\ve,2}\varphi(x)= Q_1^\ve f(x_1)+Q_1^\ve f(x_2).$ Hence one has, 
 \begin{multline*}
      \left(\int \exp\left(\frac{Q_1^\ve f}{(1+\delta)\ve(\theta_{\lambda\ve}(s)-1)}\right)\De m\right)^{\ve(\theta_{\lambda\ve}(s)-1)} \left(\int \exp\left(-\frac{f}{(1-\delta)\ve\theta_{\lambda\ve}(1-s)}\right)\De m\right)^{\ve\theta_{\lambda\ve}(1-s)}=\\\left(\int \exp\left(\frac{Q_1^{\ve,2}\varphi}{(1+\delta)\ve(\theta_{\lambda\ve}(s)-1)}\right)\De m^2\right)^{\ve(\theta_{\lambda\ve}(s)-1)/2} \left(\int \exp\left(-\frac{\varphi}{(1-\delta)\ve\theta_{\lambda\ve}(1-s)}\right)\De m^2\right)^{\ve\theta_{\lambda\ve}(1-s)/2}
  \end{multline*}
  the rest of the proof is the same as \cite[Thm.~5.1]{gozlan2017kantorovich}.

\end{proof}

\subsection{Relation with other functional inequalities}

In this section we shall see how the entropic Talagrand inequality relates to other well known functional inequalities. First, we provide a new proof via the entropic Talagrand inequality of the fact that the Logsrithmic Sobolev inequality implies Talagrand's inequality. This seminal result was first proven by Otto and Villani in \cite{otto2000generalization} . In particular, we show that 
$$
\textrm{log-Sobolev ineq.} \Rightarrow \textrm{\ei} \Rightarrow \textrm{\ti
}.$$
Our argument may be seen as a generalization to the HJB semigroup of the alternative proof of Otto and Villani's result given in \cite{bobkov2001hypercontractivity}.

\begin{cor}\label{ottovillani}
For any $\lambda,\ve,t>0$ we have the following relations 
\begin{enumerate}
    \item[(i)] If $m$ satisfies the log-Sob. inequality with constant $1/\lambda$ then it satisfies~\ei for any $s\in[0,t]$.
    \item[(ii)] If $m$ satisfies {\bf ETI}($\lambda,\ve,1$), then it satisfies \tii$(1/(2\ve\tle(1)-\ve))$.
    \item[(iii)] If the potential $U$ in $m=\exp(-2U)$ is two times continuously differentiable and $\lambda$-convex, then $m$ satisfies \eitd.
    { \item[(iv)] If $m$ satisfies~\ei for all $s\in[0,t]$}, then it satisfies the Poincar\'e inequality 
    \bes
    \forall g\in \mathcal{D}, \quad \int_{\RD} g^2(x)\, m(\De x) - \left(\int_{\RD}g(x)\,m(\De x)\right)^2 \leq \frac{2}{\lambda} \int_{\RD} |\nabla g|^2(x) \, m(\De x),
    \ees
    where $\mathcal{D}$ is the domain of the generator $\mathscr{L}=\frac{1}{2}\Delta -\nabla U\cdot \nabla $.
\end{enumerate}
\end{cor}

\begin{proof}
The statement $(i)$ is a natural consequence the equivalence between reverse hypercontractivity and the log-Sob inequality \cite[Thm.~3.3]{bakry1994hypercontractivite}.
Statement $(ii)$ follows by Remark~\ref{domination}, while statement $(iii)$ is a direct consequence of statement $(i)$. { Statement $(iv)$ is obtained following the proof of Gross' Theorem; for this reason, we do not provide full detail. We consider a bounded positive function $f$, bounded away from $0$ and with bounded derivatives of order two. The relation
\eqref{reverse-hyp} with the choices $p=0$ and $q(s)=1-\exp(\lambda\ve s)$ implies that 
$\frac{\De}{\De s}\Lambda(s)\Big|_{s=0} \geq 0$, where $\Lambda(s) = \|f\|_{q(s)}$. One obtains that 

\bes
(\| f\|_0)^{-1} \, \, \frac{\De}{\De s}\Lambda(s)\Big|_{s=0}  = -\frac{\lambda \varepsilon}{2} \left(\int_{\RD} (\log f)^2 \De m -\Big(\int_{\RD} \log f \, \De m\Big)^2\right) + \int_{\RD} \frac{1}{f} \mathscr{L}^{\varepsilon} f\, \De m
\ees
where $\mathscr{L}^{\varepsilon}$ is the generator $\frac{\ve}{2}\Delta - \ve \nabla U \cdot \nabla$. The desired conclusion is obtained by mean of some simple algebraic manipulations, the basic rules of $\Gamma$-calculus and upon setting $g=\log f$. 

}
\end{proof}

{ \begin{remark}
Adopting the notation of Bakry's notes \cite{bakry1994hypercontractivite}
and using Theorem 3.3 therein, we get that the relation
\eqref{reverse-hyp} with the choices $p=0$ and $q(s)=1-\exp(\lambda\ve s)$ implies the inequality $\mathrm{LogS}(0)$ with constant $0$. However, as it can be seen from its Definition at the bottom of page 37, such inequality is a trivial one and therefore we cannot conclude that it implies the classical Log Sobolev inequality $\mathrm{LogS(2)}$. The degeneracy is due to the fact that $q(0)=0$ and is forced by the restrictions imposed in Theorem \ref{reversehyp} on the parameters $p,q$.
\end{remark}}

Combining statements $(i)$ and $(ii)$ and taking the limit $\ve\to0$ we obtain the classical result of Otto and Villani~\cite{otto2000generalization}
$$
\textrm{log-Sobolev ineq.} \Rightarrow \textrm{\ti
}.$$

\begin{remark}
In \cite{otto2000generalization} the authors also introduce a stronger inequality which implies both the log-Sobolev and the Talagrand inequality, leading  the quadratic Wasserstein distance, the Entropy and the Fisher information together: 
$$
\ment(\mu|m) \leq W_2(\mu,m)\sqrt{I(\mu|m)}-\frac{\lambda}{2}W_2^2(\mu,m).
$$
It is interesting to point out that we can derive the entropic counterpart of this result, by differentiating in $s=0$ the convexity estimate for the entropy along Schr\"odinger bridges (see \cite[Thm.~1.4.]{conforti2017second}). 
Let us mention that an alternative proof of the classical HWI inequality is given in~\cite{gentil2018anentropic} via the Schr\"odinger problem. In particular, it is based on the Otto-Villani heuristics applied to Schr\"odinger bridges.

\end{remark}

The next result is a generalization to the entropic transportation inequality of~\cite[Thm.~2.1]{gozlan2011} in which it is introduced an inf-convolution log-Sobolev inequality that is implied by a transportation inequality with a general cost.
 
\begin{cor}[\eit and inf-convolution log-Sobolev inequality]\label{infconv}
For $\ve,\lambda>0$ and $t$ such that,
$$
1+\frac{\ve}{\exp(\lambda\ve t)-(1+\ve)} \geq0,
$$
\eit implies the following inf-convolution log-Sobolev inequality.  For any $f:\manif\to\mathbb R$,
$$
\textrm{\emph{Ent}}_m(e^f)\leq \left( 1+\frac{\ve}{\exp(\lambda\ve t)-(1+\ve)} \right)\int (f-Q_t^\ve f)e^f\De m,
$$
where we used the standard notation $\textrm{\emph{Ent}}_m(f)=\int f\log f \De m- \int f\De m\log \int f\De m.$
\end{cor}

\begin{proof}
We start by following the proof of~\cite[Thm.~2.1]{gozlan2011}. We fix $f\in C_b(\manif)$ and define $\De\nu_f=\displaystyle\frac{e^f}{\int e^f\De m} \De m$, hence we have 
\begin{align*}
\ment(\nu_f|m)&= \int \log\left(\frac{e^f}{\int e^f\De m} \right)\frac{e^f}{\int e^f\De m}\De m= \int f \De\nu_f-\log\int e^f\De m\\
& \leq \int f\De \nu_f-\int f\De m = \int (f -Q^\ve_tf) \De\nu_f +\int Q^\ve_tf \De\nu_f-\int f \De m\\
& \leq \int (f-Q^\ve_tf) \De\nu_f + \ve\ecostt(\nu_f,m)-\ve\ment(\nu_f|m)
\end{align*} 
where the first inequality is given by Jensen's inequality, while the last inequality is due to the Kantorovich dual formulation for the entropic transportation cost~\eqref{kanto}. Now~\eit implies,
$$
\ment(\nu_f|m) \leq \int (f-Q^\ve_tf) \De\nu_f + \ve\left(\tle(t)-1\right)\ment(\nu_f|m).
$$
Hence, 
$$
\ment(\nu_f|m)\left(1+\ve-\ve\tle(t) \right)\leq \int 
(f-Q^\ve_t f) \De\nu_f, 
$$
that is
$$
\ment(\nu_f|m)\left( 1-\ve(\tle(t)-1)\right)\leq \int (f-Q^\ve_tf) \De\nu_f.
$$
To conclude, we remark that $\ment(\nu_f|m)=\textrm{Ent}_m(e^f)\int e^f\De m$, thus we obtain the announced inequality.  
\end{proof}

\section{Appendix}

We briefly collect here some known and fundamental result that we made use of in the previous sections.   

\begin{lemma}[Dual representation of the entropy]\label{HJBduality1}
Let $\mathcal Z$ be a measurable space and $p \in \mathcal P(\mathcal Z)$. For any measurable function $\psi :  \mathcal Z\to [-\infty,\infty)$ it holds

\be\label{HJBduality2}
\log \int \exp(\psi) \De p= \sup \Big\{ \int \psi \De q - \ment(q|p);\, q\in \cP(\mathcal Z)\, : \int_{\mathcal Z} f_+dq<\infty \Big\} \in [-\infty,\infty]
\ee

\end{lemma}

The proof can be found for instance in  \cite{gozlan2010transport, leonard2014some}.

\begin{lemma}[Additive property of the relative entropy]\label{additiveppty} Let $\Omega, Z$ two Polish spaces. For any $p,r\in \mathcal P(\Omega)$ and any measurable function $\phi:\Omega\to Z$, 
\begin{equation}\label{eq:deco}
\ment(p|r) = \ment(p_\phi|r_\phi)+\int \ment(p(\cdot|\phi=z)|r(\cdot|\phi=z))p_\phi(dz)
\end{equation}
where $p_\phi=\phi_\# p$.

\end{lemma}
For the proof see \cite[Thm.~2.4]{leonard2014some}.




\begin{theorem}[Dual formulation of the entropic transportation cost]
$$
\ve\ecostone(\mu,\nu)=\ve \ment(\mu|m)+\sup_{\varphi\in C_b(\mathbb R^d)} \left\{ \int Q_1^\ve \varphi \,\De\mu - \int  \varphi \,\De\nu\right\}
$$
where for all $t\geq0$,
$$
Q_t^\ve\varphi(x) = -\ve\log P_t^\ve \exp(-\varphi/\ve)(x), \;\; x \in \mathbb R^d
$$
\end{theorem}

Several proofs are available \cite{MT06, gentil2015analogy, chen2014relation, gentil2018dynamical,gigli2018benamou}.

\bigskip
\bigskip
\bigskip

\noindent{\bf Acknowledgements.}  This research was partially supported  by the LABEX MILYON
ANR-10-LABX-0070. Giovanni Conforti acknowledges financial support from the CMAP, \'Ecole Polytechnique.


\bibliographystyle{plain}
\bibliography{Ref}

\end{document}